\documentclass{article}
\usepackage{amsmath,amssymb,amsthm, tikz-cd, bbm}
\usepackage[shortlabels]{enumitem}
\usepackage{graphicx}
\usepackage[margin=1in]{geometry}
\usepackage{MnSymbol, url}

\newtheorem{theorem}{Theorem}[section]
\newtheorem*{theorem*}{Theorem}

\newtheorem*{corollary*}{Corollary}
\newtheorem{lemma}[theorem]{Lemma}
\newtheorem{proposition}[theorem]{Proposition}
\newtheorem*{thmdef*}{Theorem/Definition}

\theoremstyle{definition}
\newtheorem{example}[theorem]{Example}
\newtheorem{definition}[theorem]{Definition}
\newtheorem*{definition*}{Definition}

\newtheoremstyle{iremark}
{\topsep}   
{\topsep}   
{\upshape}  
{0pt}       
{\itshape}  
{.}         
{5pt plus 1pt minus 1pt} 
{\thmname{#1}\thmnumber{ \itshape#2}\thmnote{ (#3)}} 

\theoremstyle{iremark}
\newtheorem{setup}[theorem]{Setup}
\newtheorem{convention}[theorem]{Convention}
\newtheorem{remark}[theorem]{Remark}

\newcommand{\C}{\mathbb{C}}
\newcommand{\Z}{\mathbb{Z}}

\newcommand{\la}{\langle}
\newcommand{\ra}{\rangle}
\newcommand{\lp}{\left(}
\newcommand{\rp}{\right)}
\newcommand{\sq}{\subseteq}
\newcommand{\ds}{\dots}
\newcommand{\cds}{\cdots}
\newcommand{\arrow}{arrow}

\newcommand{\N}{\mathbb{N}}

\newcommand{\fm}{\mathfrak{m}}

\renewcommand{\phi}{\varphi}

\DeclareMathOperator{\GL}{GL}

\DeclareMathOperator{\Hom}{Hom}
\DeclareMathOperator{\End}{End}

\DeclareMathOperator{\Sym}{Sym}
\DeclareMathOperator{\Spec}{Spec}

\DeclareMathOperator{\Supp}{Supp}

\DeclareMathOperator{\gr}{gr}

\DeclareMathOperator{\Der}{Der}

\DeclareMathOperator{\Char}{char}

\DeclareMathOperator{\stHom}{*Hom}

\setlist[enumerate]{itemsep=2pt, topsep=2pt, itemindent=10pt, label=(\roman*)}

\begin{document}

\title{Symmetry on rings of differential operators}
\author{Eamon Quinlan-Gallego \footnote{The author was partially supported by NSF DMS grants 1801697 and 1840234}}

\maketitle

\begin{abstract}
	If $k$ is a field and $R$ is a commutative $k$-algebra, we explore the question of when the ring $D_{R|k}$ of $k$-linear differential operators on $R$ is isomorphic to its opposite ring. Under mild hypotheses, we prove this is the case whenever $R$ Gorenstein local or when $R$ is a ring of invariants. As a key step in the proof we show that in many cases of interest canonical modules admit right $D$-module structures. 
\end{abstract}

\section{Introduction} \label{scn-intro}

Let $k$ be a field. A construction of Grothendieck assigns to every $k$-algebra $R$ its ring $D_{R|k}$ of $k$-linear differential operators, which is a noncommutative ring that consists of certain $k$-linear operators on $R$ \cite[\S 16]{EGAIV}.

Suppose $k$ has characteristic zero. If $R := k[x_1, \ds, x_n]$ is a polynomial ring over $k$ then the ring $D_{R|k}$ is called the Weyl algebra over $k$ and it has a particularly pleasant structure; for example, it is (left and right) Noetherian and its global dimension is $n$ \cite{Roos72}. These facts adapt readily to the case where $R$ is an arbitrary regular $k$-algebra that is essentially of finite type \cite[\S3]{Bjork79}, and in this context the study of $D_{R|k}$ and its modules has numerous applications in singularity theory (e.g. Bernstein-Sato polynomials) and in commutative algebra (e.g. the study of local cohomology \cite{Lyu93}). 

Since the ring $D_{R|k}$ is noncommutative, a priori its left and right modules could behave very differently. However, a key feature of the Weyl algebra $D_{R|k}$ is that it is isomorphic to its opposite ring $D_{R|k}^{op}$ via an involutive isomorphism that fixes the subring $R$ (see Remark \ref{rmk-char-zero}). This induces an equivalence between the categories of left and right $D_{R|k}$-modules, sometimes known as the side-changing functor, which is used abundantly when defining functors on $D_{R|k}$-modules \cite[\S1]{HTT}.

The goal of this paper is to show that such an isomorphism exists under much weaker hypotheses on the singularities of $R$, and even in positive characteristic. For example, our main results ensure that this is the case whenever $R$ is a normal Gorenstein local domain that is essentially of finite type over a perfect field $k$ (Theorems \ref{thm-antiauto-on-Gor} and \ref{thm-when-invol}) or whenever $R$ is a ring of invariants over a field $k$ of characteristic zero (Theorem \ref{thm-quot-invol}). We also show that such an isomorphism is unique when $k$ has positive characteristic and $R$ is a polynomial ring over $k$ (Theorem \ref{thm-curious-charp}); this is somewhat surprising, since the analogous statement in characteristic zero is far from true (see Subsection \ref{subscn-curious-result}). We note that there exist examples of rings of differential operators $D_{R|k}$ which are right Noetherian but not left Noetherian \cite[\S7]{SS88} \cite[\S5]{Muh88} \cite{Tripp}, thus providing examples where one could not have an isomorphism $D_{R|k} \cong D_{R|k}^{op}$. 

Our hope is that these results can contribute to the study of the behavior of differential operators on singular algebras, which remains somewhat mysterious despite substantial results. For example, when $X = \Spec R$ is an affine curve in characteristic zero, the structure of $D_{R|k}$ has been studied extensively and we know that $D_{R|k}$ is a (left and right) Noetherian finitely generated $k$-algebra \cite{Bloom73} \cite{Vig73} \cite{SS88} \cite{Muh88}. A famous counterexample of Bernstein, Gelfand and Gelfand showed that this does not remain true in higher dimensions: they proved that for $R := \C[x,y,z]/(x^3 + y^3 + z^3)$ the ring $D_{R|\C}$ is neither finitely generated nor Noetherian \cite{BGG72} (see \cite{Vig75} for further examples).

Despite this well-known intractability, there is renewed interest in differential operators on singular algebras due to some recent applications in commutative algebra.  For example, when $R$ is a direct summand of a polynomial or power series ring over $k$, \`{A}lvarez-Montaner, Huneke and N\'u\~nez-Betancourt have used the $D_{R|k}$-module structure to give a new proof of the finiteness of associated primes of local cohomology of $R$ \cite{AMHNB}. They also show that elements of $R$ admit Bernstein-Sato polynomials (see \cite{AMHJNBTW19} for more on this direction, including a notion of $V$-filtrations for direct summands). Brenner, Jeffries and N\'u\~nez-Betancourt have also used rings of differential operators to introduce a characteristic zero analogue of $F$-signature, called differential signature \cite{BJNB19}. 

In characteristic $p>0$ there is further evidence that $D_{R|k}$ carries information about the singularities of $R$. For example, Smith showed that an $F$-pure and $F$-finite domain $R$ is simple as a left $D_{R|k}$-module if and only if it is strongly $F$-regular \cite{Smi95}. In characteristic zero, Hsiao has shown that whenever $R$ is the homogeneous coordinate ring of a smooth projective variety $X$ the simplicity of $R$ as a left $D_{R|k}$-module implies the bigness of the tangent bundle $T_X$ of $X$ \cite{Hsi15}, and Mallory has used this to construct counterexamples to the characteristic-zero analogue of Smith's result \cite{Mall20}.
 
Let us now describe the different sections and results in the paper in more detail. The paper begins with some background material in Section 2. In Section 3 we show that if $R$ is a local or graded Cohen-Macaulay $k$-algebra then, under reasonable hypotheses, the canonical module $\omega_R$ admits a right $D_{R|k}$-module structure (Theorems \ref{thm-omegaR-rightDmodule-1},  \ref{thm-omegaR-rightDmodule-2}, \ref{thm-omegaR-rightDmodule-3} and \ref{thm-omegaR-rightDmodule-4}). This constitutes our main tool for Section 4, where we start studying the existence of isomorphisms $D_{R|k} \cong D_{R|k}^{op}$. Our main result is as follows.
\begin{theorem*} [\ref{thm-antiauto-on-Gor}]
	Let $k$ be a field and $R$ be a Gorenstein $k$-algebra. Assume that one of the following holds:
	\begin{enumerate}[(1)]
		\item The field $k$ has characteristic zero and $R$ is a local normal domain that is essentially of finite type over $k$.
		
		\item The field $k$ is a perfect field of characteristic $p>0$, and $R$ is local, $F$-finite and admits a canonical module.
		
		\item The ring $R$ is complete local and $k$ is a coefficient field.
		
		\item The ring $R = \bigoplus_{n = 0}^\infty R_n$ is graded and finitely generated over $R_0 = k$.
	\end{enumerate}
	Then there is a ring isomorphism $D_{R|k} \cong D_{R|k}^{op}$ that fixes $R$. It respects the order filtration and, if $k$ a perfect field of positive characteristic, it also respects the level filtration (see \ref{subsubscn-rodo-charp}). 
\end{theorem*}
After this work was completed we realized that some of our work from Section 3 and cases (1) and (4) from Theorem \ref{thm-antiauto-on-Gor} are already due to Yekutieli \cite[Cor. 4.9]{Yek98}, using more sophisticated machinery (most notably, \cite{Yek95}). We hope that our more elementary approach has some value.

In Section 4 we also explore the question of when such an isomorphism must be involutive (i.e. its own inverse). In this regard our result is given by the following Theorem. We note that the proof of case (1) when $\Char k = p>0$ involves the use of the Skolem-Noether theorem.
\begin{theorem*} [\ref{thm-when-invol}]
	Let $k$ be a field, $R$ be a commutative $k$-algebra and $\Phi: D_{R|k} \cong D_{R|k}^{op}$ be a ring isomorphism that fixes $R$. Then $\Phi$ is involutive in the following cases:
	\begin{enumerate}[(1)]
		\item The algebra $R$ is reduced and essentially of finite type over $k$.
		
		\item The ring $R$ is local, Gorenstein and zero dimensional, $k$ is a coefficient field of $R$ and $\Phi$ is the ring isomorphism that corresponds to the compatible right $D_{R|k}$-module structure on $R$ given by pullback via an isomorphism $R \cong \Hom_k(R,k)$ (see Lemma \ref{lemma-abstract-stuff}(a)).
	\end{enumerate}
\end{theorem*}
We conclude Section 4 by pointing out the following curious result in characteristic $p>0$, which follows from a generalization of the Skolem-Noether theorem due to Rosenberg and Zelinski \cite{RZ61} (Theorem \ref{thm-RZ}). We note that the statement is false in characteristic zero (see Subsection \ref{subscn-curious-result}). 
\begin{theorem*} [\ref{thm-curious-charp}]
	Let $k$ be a perfect field of characteristic $p>0$ and $R:=k[x_1, \ds, x_n]$ be a polynomial ring over $k$. Then there is a unique isomorphism $D_{R|k} \cong D_{R|k}^{op}$ that fixes $R$. 
\end{theorem*}
We end the paper in Section 5, where we use a theorem of Kantor (Theorem \ref{thm-Kantor}) to explore the existence of isomorphisms $D_{R|k} \cong D_{R|k}^{op}$ for rings of invariants. Our result is the following.

\begin{theorem*} [\ref{thm-quot-invol}]
	Let $k$ be a field of characteristic zero, $G$ be a finite subgroup of $\GL_n(k)$ that contains no pseudoreflections, $S := k[x_1, \ds, x_n]$ be a polynomial ring over $k$ equipped with the standard linear action of $G$ and $R := S^G$ be the ring of $G$-invariants of $S$. Then there is an involutive isomorphism $D_{R|k} \cong D_{R|k}^{op}$ that fixes $R$. 
\end{theorem*}

\subsection*{Acknowledgements}
I would like to thank my advisor, Karen Smith, for many enjoyable conversations about this topic, for her encouragement and for her support. I would also like to thank Devlin Mallory for many useful comments. Finally, I would like to thank Jack Jeffries for many useful suggestions and, in particular, for showing me how to tackle Theorem \ref{thm-omegaR-rightDmodule-1}.  This research was partially supported by NSF DMS grants 1801697 and 1840234.

\section{Background} \label{scn-background}

\subsection{Rings of differential operators} \label{subscn-rodo}

Let $k$ be a commutative ring and let $R$ be a commutative $k$-algebra. The ring $D_{R|k}$ of $k$-linear differential operators on $R$ is a particular subring of $\End_k(R)$, the ring of $k$-linear endomorphisms on $R$. After fixing some notation, we recall its definition.

\begin{convention}
	By an abuse of notation, we identify every element $f$ of $R$ with the operator $R \to R$ given by $[g \mapsto fg]$. In this manner, we also identify $R$ with the subring $\End_R(R)$ of $\End_k(R)$. 
\end{convention}

We inductively define $k$-subspaces $D^n_{R|k}$ of $\End_k(R)$, which will be called the $k$-linear differential operators $D^n_{R|k}$ of order $\leq n$ on $R$, as follows:
\begin{align}
D^0_{R|k} & = R \nonumber \\
D^n_{R|k} & = \{\xi \in \End_k(R) : [\xi, f] \in D^{n-1}_{R|k} \text{ for all } f \in R\}. \label{eqn-Dn}
\end{align}
These $D^n_{R|k}$ form an increasing chain of subspaces of $\End_k(R)$. Moreover, composing a differential operator of order at most $n$ with one of order at most $m$ yields a differential operator of order at most $n + m$; that is,
$$D^n_{R|k} D^m_{R|k} \sq D^{n+m}_{R|k}.$$
In particular, the $D^n_{R|k}$ are (left and right) $R$-submodules of $\End_k(R)$, and the increasing union
$$D_{R|k} : = \bigcup_{n = 0}^{\infty} D^n_{R|k}$$
is a subring of $\End_k(R)$, which we call the ring of $k$-linear differential operators on $R$.

\begin{remark} \label{rmk-derivations}
	We note that every $k$-linear derivation is a differential operator of order $\leq 1$; i.e. $\Der_k(R) \sq D^1_{R|k}$. In fact, one can check that $D^1_{R|k} = R \oplus \Der_k(R)$. 
\end{remark}

\begin{remark} \label{rmk-order-minus1}
	We will use the convention that $D^{-1}_{R|k} = \{0\}$, and we note that (\ref{eqn-Dn}) is still valid for $n = 0$. 
\end{remark}

\subsubsection{Rings of differential operators in positive characteristic} \label{subsubscn-rodo-charp}

Suppose now that $k$ is a perfect field of positive characteristic $p>0$ and that $R$ is $F$-finite; that is, $R$ is noetherian and finitely generated as a module over its subring $R^p$ of $p$-th powers (this holds, for example, whenever $R$ is essentially of finite type over $k$). In this setting the ring $D_{R|k}$ admits another filtration, this time by {\it subrings}, as follows.

For each positive integer $e > 0$, we define the ring $D^{(e)}_R$ of differential operators of level $e$ on $R$ by $D^{(e)}_R := \End_{R^{p^e}}(R)$; i.e. $D^{(e)}_R$ consists of all the operators on $R$ that are linear over its subring $R^{p^e}$ of $p^e$-th powers. The rings $D^{(e)}_R$ form an increasing union and we have (see \cite{Yek} \cite[\S2.5]{SVdB97})
$$D_{R|k} = \bigcup_{e = 0}^\infty D^{(e)}_R.$$
\subsubsection{Rings of differential operators on smooth algebras} \label{subsubscn-rodo-smooth}

The most typical example of a ring $D_{S|k}$ of differential operators comes from the case where $k$ is a field of characteristic zero and $S:=k[x_1, \ds, x_n]$ is a polynomial ring over $k$. In this case $D_{S|k}$ is generated by $S$ (acting, as usual, by multiplication) and the derivations $\partial_1, \ds, \partial_n$, where $\partial_i = \frac{\partial}{\partial x_i}$. The ring $D_{S|k}$ is known as the Weyl algebra in $2n$ generators over $k$ and a $k$-algebra presentation for it is given by
\begin{equation} \label{eqn-weyl-pres}
D_{S|k} = \frac{ k \la x_1, \ds, x_n,  \partial_1, \ds, \partial_n \ra }{\la [\partial_i, x_j] = \delta_{ij}, [\partial_i, \partial_j] = 0 , [x_i, x_j] = 0 \ra },
\end{equation}
where $\delta_{ij}$ is the Kronecker delta symbol. In particular, there is a left $S$-module decomposition
$$D_{S|k} = \bigoplus_{\alpha \in \N_0^n} S \partial^\alpha$$
where $\partial^\alpha = \partial_1^{\alpha_1} \cds \partial_n^{\alpha_n}$ \cite[Prop. 1.1.2]{Bjork79}.

Our next goal is to formulate a version of these statements due to Grothendieck that works, after a suitable localization, for any commutative essentially smooth algebra over an arbitrary commutative ring. More precisely, it will work under the following conditions.

\begin{setup} \label{setup-smooth}
	Let $k$ be a commutative ring and $S$ be a commutative $k$ algebra that is formally smooth and essentially of finite presentation. We furthermore assume that the module $\Omega^1_{S|k}$ of K\"ahler differentials is free over $S$ which, we recall, can always be achieved by localizing. We pick elements $x_1, \ds, x_n \in S$ such that $d x_1, \ds, d x_n$ form an $S$-basis for $\Omega^1_{S|k}$. Of course, a typical example is $S = k[x_1, \ds, x_n]$. 
\end{setup}

In order to state the result, we introduce multi-index notation. Given elements $z_1, \ds, z_n \in S$ and a multi-exponent $\alpha \in \N_0^n$ we denote by $z^\alpha$ the element $z^\alpha := z_1^{\alpha_1} \cds z_n^{\alpha_n}$. Given a  multi-exponent $\alpha \in \N_0^n$ we denote $\alpha! := \alpha_1! \cds \alpha_n!$ and $|\alpha| = \alpha_1 + \cds + \alpha_n$. Given $\alpha, \beta \in \N_0^n$, we say $\alpha \leq \beta$ whenever $\alpha_i \leq \beta_i$ for all $i = 1, 2, \ds, n$. With this notation, we have a multivariate binomial theorem: if $z_1, \ds, z_n$ and $y_1, \ds, y_n$ are elements of $S$ and $\alpha \in \N_0^n$ is a multi-exponent then $(z + y)^\alpha = \sum_{\beta + \gamma = \alpha} \frac{\alpha!}{\beta! \gamma!} z^\beta y^\gamma$. We note that, for $\beta + \gamma = \alpha$, the number $\alpha!/(\beta! \gamma!)$ is an integer. With this notation, the statement is as follows.

\begin{theorem} [{\cite[\S16.11, \S17.12.4]{EGAIV}}] \label{thm-EGA-presentation}
	Let $k$ and $S$ be as in Setup \ref{setup-smooth}. Then
	\begin{enumerate}[(1)]
		\item For every $\alpha \in \N_0^n$ there is a unique differential operator $\partial^{[\alpha]} \in D_{S|k}$ such that
		$$\partial^{[\alpha]}(x^\beta) = \frac{\beta!}{\alpha! (\beta - \alpha)!} x^{\beta - \alpha}$$
		for all $\beta \in \N_0^n$. In particular, $\partial^{[\alpha]} \partial^{[\beta]} = \frac{(\alpha + \beta)!}{\alpha! \beta!} \partial^{[\alpha + \beta]}.$
		\item We have
		$$D_{S|k} = \bigoplus_{\alpha \in \N_0^n} S \partial^{[\alpha]}.$$
	\end{enumerate}
\end{theorem}
\begin{remark} \label{rmk-char-zero}
	Suppose $k$ is a field of characteristic zero and for all $i = 1, 2, \ds, n$ let $\partial_i \in \Der_k(S)$ denote the dual of $dx_i$. Then $\partial^{[\alpha]} = (1/\alpha!) \partial^\alpha$, and Theorem \ref{thm-EGA-presentation} tells us that $D_{S|k}$ is generated by $D^1_{S|k}$ as an algebra (see Remark \ref{rmk-derivations}). When $S = k[x_1, \ds, x_n]$ is a polynomial ring over $k$, this observation allows one to recover the presentation given in (\ref{eqn-weyl-pres}). 
\end{remark} 

\subsection{Anti-automorphisms in the smooth case} \label{subscn-antiauto-smooth}

In this subsection we review the standard isomorphism between the Weyl algebra and its opposite ring, and we use Grothendieck's representation (Theorem \ref{thm-EGA-presentation}) to define an analogous operation on arbitrary smooth algebras.

We first recall some terminology. Given two (noncommutative) rings $A$ and $B$, an additive map $\Phi: A \to B$ is called a ring anti-homomorphism if $\Phi(1_A) = 1_B$ and $\Phi(xy) = \Phi(y) \Phi(x)$ for all $x, y \in A$ (equivalently, we may think of $\Phi$ as a ring homomorphism $A \to B^{op}$ or $A^{op} \to B$). If $\Phi$ is bijective we say it is an anti-isomorphism, and a self anti-isomorphism is called an anti-automorphism.

If $k$ is a field of characteristic zero and $S = k[x_1, \ds, x_n]$ is a polynomial ring over $k$ we have repeatedly mentioned that $D_{S|k}$ is a Weyl algebra over $k$, and an explicit $k$-algebra presentation is given in (\ref{eqn-weyl-pres}). Using this presentation it is easy to check that the assignments $[x_i \mapsto x_i]$ and $[\partial_i \mapsto - \partial_i]$ define an anti-automorphism on $D_{S|k}$, usually called the standard transposition. Note that the standard transposition is involutive (i.e. its own inverse) and that it fixes the subring $S$. In this subsection, we extend this to the case of smooth algebras (see Proposition \ref{prop-exists-involn-smoothcase}).

\begin{lemma} \label{lemma-bin-coeff}
	Let $\sigma \in \N_0^n$ be a nonzero multi-exponent. Then
	$$\sum_{\beta + \delta = \sigma} (-1)^{|\beta|} \frac{\sigma!}{\beta! \delta!} = 0.$$
\end{lemma}
\begin{proof}
	Let $y_1 = \cds = y_n = 1$, $z_1 = \cds = z_n = -1$ and apply the multivariate binomial theorem to conclude $0 = (z + y)^{\sigma} = \sum_{\beta + \delta = \sigma} \frac{\sigma!}{\beta! \delta!} z^\beta y^\delta = \sum_{\beta + \delta = \sigma} (-1)^{|\beta|} \frac{\sigma!}{\beta! \delta!}$. 
\end{proof}

\begin{lemma} \label{lemma-commute-partial-f}
	Let $k$ and $S$ be as in Setup \ref{setup-smooth}. Then for all $\alpha \in \N_0^n$ and $f \in S$ we have the following equality in $D_{S|k}$:
	$$\partial^{[\alpha]} f = \sum_{\substack{\beta, \gamma \in \N_0^n \\ \beta + \gamma = \alpha}} \partial^{[\beta]}(f) \partial^{[\gamma]}. $$
\end{lemma}
\begin{proof}
	Fix some integer $m \geq |\alpha|$. We let $J \sq R \otimes_k R$ denote the ideal of the multiplication map $\mu: R \otimes_k R \to R$, and we let $P^m_{R|k} := R \otimes_k R / J^{m+1}$ denote the $m$-th module of principal parts for $R$ over $k$, so that we have an identification $\Hom_S(P^m_{S|k}, S) \cong D^m_{S|k}$, given by $[\phi \mapsto [g \mapsto \phi(1 \otimes g)]]$ \cite[\S16]{EGAIV}. Given $g \in R$ we denote by $d^mg$ the element $d^mg := 1 \otimes g - g \otimes 1 \in P^m_{R|k}$. Then $P^m_{R|k}$ is, as a left $R$-module, free in the basis $\{(d^m x)^\beta : |\beta| \leq m \}$, and $\partial^{[\alpha]}$ is, thought of as an $R$-linear map $\partial^{[\alpha]}: P^n_{R|k} \to R$, the dual of the basis element $(dx)^\alpha$ \cite[\S16]{EGAIV}. In particular, for all $g \in S$ we have an equality $1 \otimes g = \sum_{|\beta| \leq n} \partial^{[\beta]}(g) (d^m x)^\beta$ in $P^m_{S|k}$. 
	
	We need to show that, for all $g \in S$, we have $\partial^{[\alpha]}(fg) = \sum_{\beta + \gamma = \alpha} \partial^{[\beta]}(f) \partial^{[\gamma]}(g)$. This follows by considering the following equalities in $P^m_{R|k}$
	\begin{align*}
	1 \otimes fg & = (1 \otimes f)(1 \otimes g) \\
	& = (\sum_{|\beta| \leq m} \partial^{[\beta]})(f) (d^m x)^\beta) (\sum_{|\gamma| \leq m} \partial^{[\gamma]}(f) (dx)^\gamma) 
	\end{align*}
	and extracting the coefficient of $(d^m x)^\alpha$ of the last expression. 
\end{proof}

\begin{proposition} \label{prop-exists-involn-smoothcase}
	Let $k$ and $S$ be as in Setup \ref{setup-smooth}. Then there is a unique involutive anti-automorphism on $D_{S|k}$ that fixes $S$ and sends $[\partial^{[\alpha]} \mapsto (-1)^{|\alpha|} \partial^{[\alpha]}]$ for all $\alpha \in \N_0^n$.
\end{proposition}
\begin{proof}
	By Theorem \ref{thm-EGA-presentation}, $D_{S|k}$ is free as a left $S$-module with basis $\{\partial^{[\alpha]}: \alpha \in \N_0^n\}$. Therefore there is a unique additive map $\Phi: D_{S|k} \to D_{S|k}$ with $\Phi(f \partial^{[\alpha]}) = (-1)^{|\alpha|} \partial^{[\alpha]} f$ for all $f \in S$ and $\alpha \in \N_0^n$. We want to show that $\Phi$ is a ring anti-homomorphism.
	
	One immediately verifies that, for all $\xi \in D_{S|k}, \alpha \in \N_0^n$ and $f \in S$ we have $\Phi(\xi \partial^{[\alpha]}) = \Phi(\partial^{[\alpha]}) \Phi(\xi)$ and $\Phi(f \xi) = \Phi(\xi) f$. We conclude it suffices to show that $\Phi(\partial^{[\alpha]} f) = f \Phi(\partial^{[\alpha]})$. By using Lemma \ref{lemma-commute-partial-f} we conclude that
	\begin{align*}
	\Phi(\partial^{[\alpha]} f) & = \Phi(\sum_{\beta + \gamma = \alpha}\partial^{[\beta]}(f) \partial^{[\gamma]} ) \\
		& = \sum_{\beta + \gamma = \alpha} (-1)^{|\gamma|} \partial^{[\gamma]} \partial^{[\beta]}(f) \\
		& = \sum_{\beta + \gamma = \alpha} (-1)^{|\gamma|} \sum_{\delta + \epsilon = \gamma} \partial^{[\delta]}( \partial^{[\beta]}(f)) \partial^{[\epsilon]} \\
		& = \sum_{\beta + \delta + \epsilon = \alpha} (-1)^{|\delta + \epsilon|} \frac{(\beta + \delta)!}{\beta! \delta!} \partial^{[\delta + \beta]} (f) \partial^{[\epsilon]} \\
		& = (-1)^{|\alpha|} \sum_{\epsilon \leq \alpha} \lp \sum_{\beta + \delta = \alpha - \epsilon} (-1)^{|\beta|} \frac{(\alpha - \epsilon)!}{\beta! \delta!} \rp \partial^{[\alpha - \epsilon](f) \partial^{[\epsilon]}} \\
		& = (-1)^{|\alpha|} f \partial^{[\alpha]},
	\end{align*}
	where in the last equality we made use of Lemma \ref{lemma-bin-coeff}. This completes the proof of the fact that $\Phi$ is a ring anti-homomorphism. To show that it is an involution we now simply observe that $\Phi^2$ is a ring homomorphism with $\Phi^2(f) = f$ and $\Phi^2(\partial^{[\alpha]}) = \partial^{[\alpha]}$ for every $f \in S$ and $\alpha \in \N_0^n$, and therefore $\Phi^2$ must be the identity.
\end{proof}

\section{Right $D$-module structures on canonical modules} \label{scn-rightD-on-can}

Let $k$ be a field and $R$ be a local $k$-algebra. In this section we show that, under reasonable hypotheses, the canonical module $\omega_R$ carries a compatible right $D_{R|k}$-module structure. This fact will be the main ingredient in our construction of anti-automorphisms on rings of differential operators on Gorenstein algebras. Before we begin, let us clarify  what ``compatible" means in this setting.
\begin{definition} \label{def-comp}
	Let $M$ be an $R$-module. A (right or left) $D_{R|k}$-module structure on $M$ is compatible if it extends the already-existing $R$-module structure. 
\end{definition}

We note that Theorem \ref{thm-omegaR-rightDmodule-1} and Theorem \ref{thm-omegaR-rightDmodule-4} follow from the duality for differential operators developed by Yekutieli \cite[Cor. 4.2]{Yek98}, or by noting that the co-stratification approach of Berthelot \cite[\S 1.2.1]{BerthelotII} (in the case $m = \infty$) also works in the singular setting. Trying to keep an elementary exposition, we prove these results using simpler techniques.

\subsection{Characteristic zero case}

\begin{theorem} \label{thm-omegaR-rightDmodule-1}
	Let $k$ be a field of characteristic zero and $R$ be a local Cohen-Macaulay normal domain that is essentially of finite type over $k$. Then the canonical module $\omega_R$ admits a compatible right $D_{R|k}$-module structure. 
\end{theorem}
\begin{proof}
	Let $U \sq \Spec R$ denote the smooth locus of $R$; by the normality of $R$, its complement has codimension $\geq 2$. By use of the Lie derivative, we have an action of the tangent sheaf $T_U$ on the sheaf of modules $\omega_R|_U$, which extends to a compatible right module structure over the sheaf of differential operators on $U$ \cite[\S 1.2]{HTT} (note: one crucially uses that $k$ has characteristic zero here, since otherwise the tangent sheaf does not generate the sheaf of differential operators). We conclude that $\Gamma(U, \omega_R)$ has a compatible right $\Gamma(U, D_R)$-module structure. 
	
	The module $\omega_R$ is reflexive \cite{Aoyama83} \cite[0AVB]{Stacks}. Since we have an $R$-module isomorphism $D^n_{R|k} \cong \Hom_R (P^n_{R|k}, R)$, where $P^n_{R|k}$ is the $n$-th module of principal parts for $R$ \cite[\S16]{EGAIV}, we conclude that $D^n_{R|k}$ is also a reflexive $R$-module \cite[0AV6, 0AVB]{Stacks}. Therefore, the natural maps $\omega_R \to \Gamma(U, \omega_R)$ and $D_R \to \Gamma(U, D_R)$ are isomorphisms. It follows that $\omega_R$ has a compatible right $D_R$-module structure, as required.
\end{proof}
\begin{remark}
	 Note that neither the Cohen-Macaulay nor the local hypotheses were really needed in the proof of Theorem \ref{thm-omegaR-rightDmodule-1}: if we remove them, the only issue is that we need to make sense of what $\omega_R$ is. Since $R$ is normal, the smooth locus $U \sq \Spec R$ has codimension at least two, and if we define $\omega_R$ to be the unique reflexive extension of the canonical bundle on $U$ then the proof given goes through verbatim.
\end{remark}

\subsection{Characteristic $p > 0$ case}

We will tackle the case of characteristic $p>0$ next, and we begin by setting up some notation.

Let $k$ be a perfect field of characteristic $p > 0$ and $R$ be an $F$-finite $k$-algebra. We denote by $F: R \to R$ the Frobenius morphism (i.e. $F(r) = r^p$). Given a positive integer $e$, we denote by $F^e$ the $e$-th iteration of $F$, and we let $R^{p^e}$ be the subring of $p^e$-th powers of $R$. Given $e$ we also define the set $F^e_* R: = \{F^e_* r : r \in R\}$, which we endow with abelian group structure $F^e_* r + F^e_* s = F^e_* (r + s)$ and an $(R, R)$-bimodule structure given by $f \cdot F^e_* r = F^e_* f^{p^e} r$ and $(F^e_*r)\cdot f := F^e_*R(rf)$ --- i.e. as a right $R$-module $F^e_*R $ is (isomorphic to) $R$, and as a left $R$-module $F^e_* R$ is (isomorphic to) the restriction of scalars of $R$ across $F^e$.

Recall that in our setting the ring $D_{R|k}$ of $k$-linear differential operators on $R$ is given by $D_{R|k} = \bigcup_{e = 0}^\infty D^{(e)}_R,$ where $D^{(e)}_R = \End_{R^{p^e}}(R)$ (see \ref{subsubscn-rodo-charp}). The inclusion $R^{p^e} \sq R$ and the Frobenius map $R \to F^e_* R$ $[r \mapsto F^e_* r^{p^e}]$ are isomorphic, in the sense that the diagram
$$\begin{tikzcd}
R \arrow[r, "r \mapsto F^e_* r^{p^e}"] \arrow[d, "r \mapsto r^{p^e}"] & F^e_* R \arrow[d, "F^e_* r \mapsto r"]\\
R^{p^e} \arrow[r, hook] & R . 
\end{tikzcd}$$
commutes, and we can thus identify $D^{(e)}_{R}$ with
$$D^{(e)}_{R} \cong \End_R (F^e_* R),$$
where $\End_R(F^e_* R)$ is the endomorphism ring of $F^e_* R$ as a left $R$-module. 

If $M$ is an $R$-module we let
$$(F^e)^\flat M := \Hom_R(F^e_* R, M)$$
be the set of left $R$-linear homomorphisms from $F^e_* R$ to $M$, endowed with an $R$-module structure given by 
$$(f \cdot \varphi)(F^e_* r) := \varphi\big(F^e_* (rf) \big) \hspace*{15pt} (\varphi \in (F^e)^\flat M \text{ and }  r, f \in R);$$
that is, the structure inherited from the right $R$-module structure on $F^e_* R$. This construction respects morphisms and therefore induces functors $(F^e)^\flat$ on $R$-modules. 

We note that for all $e, d> 0$ there are bimodule isomorphisms $F^d_* R \otimes_R F^e_* R \cong F^{d + e}_* R$ given by $[F^d_* a \otimes F^e_* b \mapsto F^{d + e}_* a^{p^e} b]$ which, by the tensor-Hom adjunction, give isomorphisms
$$(F^{d })^\flat (F^e)^\flat M \cong (F^{d + e})^\flat M$$
which induce natural equivalences of functors.

\begin{theorem} \label{thm-omegaR-rightDmodule-2}
	Let $k$ be a perfect field of characteristic $p>0$ and $R$ be a local Cohen-Macaulay $F$-finite $k$-algebra with canonical module $\omega_R$. Then $\omega_R$ admits a compatible right $D_{R|k}$-module structure. 
\end{theorem}
\begin{proof}
	There is an $R$-module isomorphism $\omega_R \cong F^\flat \omega_R$ \cite[Thm. 3.3.7(b)]{BrunsHerzog} which, in turn, induces isomorphisms $\omega_R \cong (F^{e})^\flat \omega_R$ for every $e>0$. By the description of $(F^e)^\flat$, we observe that $(F^e)^\flat \omega_R$ --- and, indeed, any $(F^e)^\flat M$ --- inherits a natural right-module structure over the ring $\End_R (F^e_* R)$, and thus over the ring $D^{(e)}_R$. Since the isomorphisms $(F^e)^\flat \omega_R \to (F^{e + 1})^\flat \omega_R$ are $D^{(e)}_R$-linear, the module
	$$\omega_R \cong \lim_{\to e} (F^e)^\flat \omega_R$$
	inherits a right $D_R$-module structure. Since the maps $\omega_R \to (F^e)^\flat \omega_R$ are $R$-module maps, this right $D_R$-module structure is compatible. 
\end{proof}
\begin{remark}
	With the appropriate replacements for $\omega_R$ we do not need the local hypothesis in Theorem \ref{thm-omegaR-rightDmodule-2}. Our proof shows that if $R$ is equidimensional and essentially of finite type over $k$, $i: \Spec R \to \Spec k$ is the structure map and $\omega_R = i^!k [-\dim R]$ then $\omega_R$ carries a natural compatible right $D_{R|k}$-module structure.
\end{remark}

\subsection{Complete case}

We now tackle the complete case, which is an easy application of Matlis duality for $D$-modules. This duality was developed by Switala (similar constructions appear in the work of Yekutieli \cite{Yek95}). We will quickly summarize the key facts, and we refer to Switala's work \cite[\S 4]{Switala17} or thesis \cite{Swi} for details. 

Let $R$ be a Cohen-Macaulay complete local ring with coefficient field $k$. Let $\fm$ be the maximal ideal of $R$ and, by an abuse of notation, we will denote the residue field $R/\fm$ by $k$. Given a (not necessarily finitely generated) $R$-module $M$, a $k$-linear map $\phi: M \to k$ is called $\Sigma$-continuous if for every $u \in M$ there exists some $s > 0$ such that $\phi(\fm^s u) = 0$. We denote by $\Hom_k^\Sigma(M,k)$ the collection of all such maps, which has an $R$-module structure by premultiplication. We note that if $M$ is such that $\Supp(M) = \{\fm\}$ then every $k$-linear map is $\Sigma$-continuous.

Let $E = E_R(k)$ be the injective envelope of $k$, and fix a $k$-linear splitting $\sigma$ of the inclusion $k \hookrightarrow E$. Given an arbitrary $R$-module $M$, the map
\begin{align*}
\Hom_R (M, E) & \rightarrow  \Hom^\Sigma_k(M,k) \\
\psi & \mapsto  \sigma \circ \psi
\end{align*}
gives an isomorphism between $\Hom_R(M, E)$ (the Matlis dual of $M$) and the module $\Hom_k^\Sigma(M,k)$ which, in fact, gives a natural equivalence of functors. 

\begin{theorem} \label{thm-omegaR-rightDmodule-3}
	Let $R$ be a Cohen-Macaulay complete local ring with coefficient field $k$. Then the canonical module $\omega_R$ of $R$ admits a compatible right $D_{R|k}$-module structure.  
\end{theorem}
\begin{proof}
	Since $R$ is complete, the canonical module $\omega_R$ is the Matlis dual of the top local cohomology module $H^d_\fm (R)$ and, since $\Supp H^d_\fm (R) = \{\fm\}$, we have
	$$\omega_R = \Hom_k(H^d_\fm( R), k).$$
	Since $H^d_\fm (R)$ has a compatible left $D_{R|k}$-module structure (via the \v{C}ech complex), $\omega_R$ has a compatible right $D_{R|k}$-module structure, given by premultiplication. 
\end{proof}
\subsection{Graded case}
Finally, we consider the graded case. The strategy is identical to the one for the complete case, except in that we use graded versions of all the relevant constructions. 

Suppose that $R = \bigoplus_{n = 0}^\infty R_n$ is a graded ring which is finitely generated over a field $k = R_0$. Whenever we talk about the canonical module of such a graded ring $R$ we will always mean the *canonical module in the sense of Bruns and Herzog \cite[\S3.6]{BrunsHerzog}. In particular, it is a graded module which is uniquely determined up to graded isomorphism \cite[Prop.3.6.9]{BrunsHerzog}. 

Given a graded $R$-module $M = \bigoplus_{i \in \Z} M_i$, its graded Matlis dual is given by
$$\stHom_k (M, k) = \bigoplus_{i \in \Z} \Hom_k(M_{-i}, k),$$
which has a natural structure of graded $R$-module. Whenever the module $M$ has a left $D_{R|k}$-module structure, its graded Matlis dual acquires a compatible right $D_{R|k}$-module structure by precomposition \cite{SwiZh} \cite[Prop. 3.1]{Jeffries-derived_functors}.
\begin{theorem} \label{thm-omegaR-rightDmodule-4}
	Let $R = \bigoplus_{n = 0}^\infty R_n$ be a Cohen-Macaulay graded ring which is finitely generated over a field $k = R_0$. Then the canonical module $\omega_R$ admits a compatible right $D_R$-module structure. 
\end{theorem}
\begin{proof}
	The canonical module $\omega_R$ is the graded Matlis dual of the top local cohomology module $H^d_\fm (R)$ \cite[Thm. 3.6.19]{BrunsHerzog}. Since $H^d_\fm(R)$ is a left $D_{R|k}$-module (via the \v{C}ech complex), $\omega_R$ acquires a compatible right $D_{R|k}$-module structure. 
\end{proof}

\section{Symmetry on rings of differential operators} \label{scn-symmetry}

\subsection{Existence of anti-automorphisms on $D_{R|k}$ for Gorenstein algebras}

The goal of this section is to prove Theorem \ref{thm-antiauto-on-Gor}. We begin with a few lemmas that hold for arbitrary commutative rings $k$ and $R$. 

\begin{lemma} \label{lemma-mult_is_diff_op}
	Let $k$ be a commutative ring and $R$ be a commutative $k$-algebra. Suppose that $R$ carries a compatible right $D_{R|k}$-module structure. If $\xi \in D_{R|k}$ has order $\leq n$ then the $k$-linear operator $\xi^\ast$ given by $\xi^\ast(f) = f \cdot \xi$ for all $f \in R$ is also a differential operator of order $\leq n$.
\end{lemma} 
\begin{proof}
	We prove this by induction on $n$, with the case $n = 0$ following from the compatibility of the right $D_{R|k}$-module structure. Therefore, suppose that the statement holds for $n = m - 1$, and let $f \in R$. We observe that $\xi^* f = (f \xi)^*$ and $f \xi^* = (\xi f)^*$, from which it follows that 
	$$[\xi^*, f] = - [\xi, f]^*.$$
	Since $\xi$ has order $\leq m$, $[\xi, f]$ has order $\leq m-1$  which, together with the inductive hypothesis, implies that $[\xi^*, f]$ has order $\leq m -1$. Since $f \in R$ was arbitrary, we conclude that $\xi^*$ has order $\leq m$, as required. 
\end{proof}

\begin{lemma} \label{lemma-homo-is-iso}
	Let $k$ be a commutative ring, $R$ be a commutative $k$-algebra and $\Phi: D_{R|k} \to D_{R|k}$ be a ring anti-homomorphism that fixes $R$. Then $\Phi$ is an anti-isomorphism. 
\end{lemma}
\begin{proof}
	By Lemma \ref{lemma-mult_is_diff_op}, $\Phi$ respects the order filtration. We now claim that if $\xi \in D_{R|k}$ has order $\leq n$ then the operator $\Phi(\xi) + (-1)^{n+1} \xi$ has order $\leq n - 1$. The claim being clear for $n = 0$ (see Remark \ref{rmk-order-minus1}), we assume that it is true for $n = m-1$. Suppose that $\xi$ has order $\leq m$ and let $f \in R$. We observe that
	\begin{align*}
	[\Phi(\xi) + (-1)^{m+1} \xi, f] & = [\Phi(\xi), f] + (-1)^{m+1} [\xi, f] \\	
	& = - \Phi([\xi, f]) + (-1)^{m+1} [\xi, f] \\
	& = - \lp \Phi([\xi, f]) + (-1)^m [\xi, f] \rp
	\end{align*}
	and we therefore conclude that $[\xi + (-1)^{n+1} \Phi(\xi), f]$ has order $\leq m - 2$. As $f \in R$ was arbitrary, $\xi + (-1)^{n+1} \Phi(\xi)$ has order $\leq m - 1$, which proves the claim. 
	
	The claim implies that the associated graded map $\gr(\Phi): \gr D_{R|k} \to \gr D_{R|k}$ is given by multiplication by $-1$ on odd degrees and the identity on even degrees and that, in particular, it is an isomorphism (note that $\gr D_{R|k}$ is a commutative ring). We conclude that $\Phi: D_{R|k} \to D_{R|k}$ is bijective as well. 	
\end{proof}

\begin{lemma} \label{lemma-abstract-stuff}
	Let $k$ be a commutative ring and $R$ be a commutative $k$-algebra. 
	\begin{enumerate}[(a)]
		\item There is a one-to-one correspondence between the compatible right $D_{R|k}$-module structures on $R$ and the anti-automorphisms on $D_{R|k}$ that fix $R$.
		\item Every anti-automorphism on $D_{R|k}$ that fixes $R$ respects the order filtration.
		\item If $k$ is a perfect field of positive characteristic and $R$ is $F$-finite then every anti-automorphism on $D_{R|k}$ that fixes $R$ respects the level filtration.
	\end{enumerate}
\end{lemma}
\begin{proof}
	We begin with part (a), for which we explicitly describe the correspondence. Suppose $\Phi: D_{R|k} \to D_{R|k}$ is a ring anti-automorphism that fixes $R$. The ring $R$ is naturally a left $D_{R|k}$-module, and by restricting scalars across $\Phi$ we give it a right $D_{R|k}$-module structure. More explicitly, we define $f \cdot \xi := \Phi(\xi) (f)$ for all $\xi \in D_{R|k}$ and $f \in R$. Because $\Phi$ fixes $R$, this right $D_{R|k}$-module structure is compatible. 
	
	Conversely, suppose $R$ has a compatible right $D_{R|k}$-module structure. Given $\xi \in D_{R|k}$ we define $\xi^\ast$ to be the $k$-linear operator defined by $\xi^\ast (f) = f \cdot \xi$ for all $f \in R$ (i.e. $\xi^\ast$ is multiplication by $\xi$ under the given right $D_{R|k}$-module structure). By Lemma \ref{lemma-mult_is_diff_op}, $\xi^\ast$ is a differential operator and therefore the assignment $[\xi \mapsto \xi^\ast]$ gives a ring anti-homomorphism $\Phi: D_{R|k} \to D_{R|k}$ that fixes $R$. By Lemma \ref{lemma-homo-is-iso}, $\Phi$ is an anti-automorphism. 
	
	These two constructions are readily checked to be inverses, which proves part (a).
	
	We now tackle (b). Given a ring anti-automorphism $\Phi: D_{R|k} \to D_{R|k}$ that fixes $R$, we consider the compatible right $D_{R|k}$-module structure on $R$ given by restriction of scalars. The result then follows from Lemma \ref{lemma-mult_is_diff_op} (alternatively, one could use an induction argument entirely analogous to the one in Lemma \ref{lemma-mult_is_diff_op}).
	
	We now assume that $k$ is perfect of characteristic $p>0$ and that $R$ is $F$-finite, and we prove part (c). Suppose that $\Phi: D_{R|k} \to D_{R|k}^{op}$ is an anti-automorphism that fixes $R$. Recall that $\xi \in D_{R|k}$ has level $e>0$ if and only if $\xi$ is $R^{p^e}$-linear; i.e. if $\xi f^{p^e} = f^{p^e} \xi$ for all $f \in R$. Applying $\Phi$ to such an equality we obtain $f^{p^e} \Phi(\xi) = \Phi(\xi) f^{p^e}$, which proves the result.
\end{proof}
 
\begin{theorem} \label{thm-antiauto-on-Gor}
	Let $k$ be a field and $R$ be a Gorenstein $k$-algebra. Assume that one of the following holds:
	\begin{enumerate}[(1)]
		\item The field $k$ has characteristic zero and $R$ is a local normal domain that is essentially of finite type over $k$.
		
		\item The field $k$ is a perfect field of characteristic $p>0$, and $R$ is local, $F$-finite and admits a canonical module.
		
		\item The ring $R$ is complete local and $k$ is a coefficient field.
		
		\item The ring $R = \bigoplus_{n = 0}^\infty R_n$ is graded and finitely generated over $R_0 = k$. 
	\end{enumerate}
	Then there is a ring anti-automorphism on $D_{R|k}$ that fixes $R$. It respects the order filtration and, if $k$ a perfect field of positive characteristic, it also respects the level filtration. 
\end{theorem}
\begin{proof}
By Theorems \ref{thm-omegaR-rightDmodule-1}, \ref{thm-omegaR-rightDmodule-2}, \ref{thm-omegaR-rightDmodule-3} and \ref{thm-omegaR-rightDmodule-4} we know that $\omega_R$ admits a compatible right $D_{R|k}$-module structure. By pulling it back across an $R$-module isomorphism $R \cong \omega_R$, we obtain a compatible right $D_{R|k}$-module structure on $R$ which, by Lemma \ref{lemma-abstract-stuff} (a), corresponds to a ring automorphism $D_{R|k} \cong D_{R|k}^{op}$ that fixes $R$ (explicitly, the ring automorphism is given by $[\xi \mapsto \xi^\ast]$ where $\xi^\ast(f) = f \cdot \xi$ for all $\xi \in D_{R|k}$ and $f \in R$). The rest of the statements follow from Lemma \ref{lemma-abstract-stuff} (b) and (c).
\end{proof}
As mentioned in Section \ref{scn-intro}, after this work was completed we realized that the cases (1), and (4) of Theorem \ref{thm-antiauto-on-Gor} were already proven by Yekutieli \cite[Cor. 4.9]{Yek98} in greater generality, albeit using more sophisticated methods.

\subsection{Involutivity of anti-automorphisms on $D_{R|k}$}

We now address the question of when a ring anti-automorphism $\Phi: D_{R|k} \xrightarrow{\sim} D_{R|k}$ on $D_{R|k}$ is involutive (i.e. its own inverse). For starters, we note that the associated graded morphism $\gr (\Phi) : \gr D_{R|k} \to \gr D_{R|k}$ is always involutive (see Proof of Lemma \ref{lemma-homo-is-iso}), and that the following simple observation allows us to reduce the problem significantly.

\begin{lemma} \label{lemma-suff-to-check-on-one}
	Let $\Phi: D_{R|k} \xrightarrow{\sim} D_{R|k}$ be a ring anti-automorphism that fixes $R$. If for all $\xi \in D_{R|k}$ we have $\Phi^2(\xi)(1) = \xi(1)$ then $\Phi$ is involutive. 
\end{lemma}
\begin{proof}
	For such a $\Phi$ we have $\Phi^2(\xi)(f) = (\Phi^2(\xi) f)(1) = \Phi^2(\xi f)(1) = (\xi f)(1) = \xi(f)$ for all $f \in R$ and $\xi \in D_{R|k}$, and thus $\Phi^2(\xi) = \xi$
\end{proof}

We can also show that $\Phi$ is involutive on differential operators of order $\leq 1$.

\begin{lemma} \label{lemma-der-inv-ok}
	Let $k$ be a commutative ring, $R$ be a commutative $k$-algebra and $\Phi$ be a ring anti-automorphism on $D_{R|k}$ that fixes $R$. If $\xi$ is a differential operator of order $\leq 1$ then $\Phi^2(\xi) = \xi$. Consequently, if $D_{R|k}$ is generated as an algebra by the differential operators of order $\leq 1$ then $\Phi$ must be involutive.
\end{lemma}
\begin{proof}
	We recall that $D^1_{R|k} = R \oplus \Der_k(R)$ and, therefore, it suffices to prove the claim in the case where $\xi$ is a derivation. With this assumption, we observe that for all $f \in R$ we have 
	\begin{align*}
	\Phi(\xi)(f) & = (\Phi(\xi) f)(1) \\	
		& = ([\Phi(\xi), f] + f \Phi(\xi)(1) \\
		& = (- \Phi([\xi, f]) + f \Phi(\xi))(1) \\
		& = - \xi(f) + f \Phi(\xi)(1)
	\end{align*}
	and therefore $\Phi(\xi) = - \xi + \Phi(\xi)(1)$. We conclude that $\Phi^2(\xi) = \Phi(-\xi + \Phi(\xi)(1)) = - \Phi(\xi) + \Phi(\xi)(1) = \xi$ as required.
\end{proof}
It is known that whenever $k$ is a field of characteristic zero and $R$ is regular and essentially of finite type over $k$ the ring $D_{R|k}$ is generated by the differential operators of order $\leq 1$ \cite[\S15.5.6]{McCRob}, and therefore any ring anti-automorphism on $D_{R|k}$ that fixes $R$ is involutive. In characteristic $p>0$, we use the Skolem-Noether theorem to prove an analogous result for field extensions.
\begin{lemma} \label{lemma-K-involutive-charp}
	Let $k$ be a perfect field of characteristic $p>0$ and $K$ be a finitely generated field extension of $k$. Then every ring anti-automorphism $\Phi$ on $D_{K|k}$ that fixes $K$ is involutive.
\end{lemma}
\begin{proof}
	By Proposition \ref{prop-exists-involn-smoothcase} we know there is an involutive ring anti-automorphism $\Psi$ on $D_{K|k}$ that fixes $K$. By Lemma \ref{lemma-abstract-stuff}(c) both $\Phi$ and $\Psi$ respect the level filtration (see Subsection \ref{subscn-rodo}), and it suffices to show that for every $e>0$ the restriction of $\Phi$ to the subring of differential operators of level $e$ is involutive. We thus fix an arbitrary $e > 0$.
	
	We observe that $D^{(e)}_K = \End_{K^{p^e}}(K)$ is a matrix ring over the field $K^{p^e}$. The Skolem-Noether theorem tells us that every automorphism of a matrix algebra over a field is inner and, in particular, so is the composition
	$$\begin{tikzcd}
	D^{(e)}_K \arrow[r, "\Psi"] & D^{(e)}_K \arrow[r, "\Phi"] & D^{(e)}_K.
	\end{tikzcd}$$
	We conclude that there exists some invertible $\gamma \in D^{(e)}_K$ such that $\Phi \Psi (\xi) = \gamma \xi \gamma^{-1}$ for all $\xi \in D^{(e)}_K$ and, replacing $\xi$ by $\Psi(\xi)$, we conclude that $\Phi(\xi) = \gamma \Psi(\xi) \gamma^{-1}$ for all $\xi \in D^{(e)}_K$. 
	
	We note that, since $\Phi$ and $\Psi$ fix $K$, we have that $f \gamma = \gamma f$ for all $f \in K$; that is, $\gamma \in D^{(0)}_K = K$. In particular, $\gamma$ and $\gamma^{-1}$ are also fixed by both $\Phi$ and $\Psi$. We conclude that for all $\xi \in D^{(e)}_K$ we have
	\begin{align*}
	\Phi^2(\xi) & = \Phi (\gamma \Psi(\xi) \gamma^{-1}) \\
		& = \gamma^{-1} \Phi \Psi(\xi) \gamma \\
		& = \gamma^{-1} (\gamma \Psi^2 (\xi) \gamma^{-1}) \gamma \\
		& = \xi 
	\end{align*}
	as required.
\end{proof}

We are now ready to prove our main result in this subsection.

\begin{theorem} \label{thm-when-invol}
	Let $k$ be a field, $R$ be a commutative $k$-algebra and $\Phi$ be a ring anti-automorphism on $D_{R|k}$ that fixes $R$. Then $\Phi$ is involutive in the following cases:
	\begin{enumerate}[(1)]
		\item The algebra $R$ is reduced and essentially of finite type over $k$.
		
		\item The ring $R$ is local, Gorenstein and zero dimensional, $k$ is a coefficient field of $R$ and $\Phi$ is the ring anti-automorphism that corresponds to the compatible right $D_{R|k}$-module structure on $R$ given by pullback via an isomorphism $R \cong \Hom_k(R,k)$ (see Lemma \ref{lemma-abstract-stuff}(a)).
	\end{enumerate}
\end{theorem}

\begin{proof}
	We begin with case (1). First observe that the anti-automorphism $\Phi$ on $D_{R|k}$ gives $R$ a compatible right $D_{R|k}$-module structure, which induces a compatible right $D_{W^{-1}R |k}$-module structure on every localization $W^{-1} R$ of $R$ \cite[\S16.4.14]{EGAIV}, which in turn corresponds to an anti-automorphism on $D_{W^{-1}R|k}$ that fixes $W^{-1} R$ (see Lemma \ref{lemma-abstract-stuff}(a)). In this way, the isomorphism $\Phi$ induces an isomorphism of sheaves of rings on $\Spec R$. It then suffices to show that $\Phi$ is involutive locally and we therefore restrict to the case where $R$ is a domain. Moreover, by applying the same construction to the field of fractions $K$ of $R$ we obtain an anti-automorphism on $D_{K|k}$ in such a way that the diagram
	$$\begin{tikzcd}
	D_{K|k} \arrow[r, "\sim"] & D_{K|k} \\
	D_{R|k} \arrow[u, hook] \arrow[r, "\sim"] & D_{R|k} \arrow[u, hook]
	\end{tikzcd}$$
	commutes, and therefore it suffices to show that the anti-automorphism on $D_{K|k}$ is involutive. When $\Char k = p>0$ this was already proved in Lemma \ref{lemma-K-involutive-charp}. When $\Char k = 0$ the ring $D_{K|k}$ is generated by the elements of order $\leq 1$, since $K$ is regular and essentially of finite type over $k$, and the statement then follows from Lemma \ref{lemma-der-inv-ok}. 
	
	Let us now tackle case (2). Let $\varphi: R \xrightarrow{\sim} \Hom_k(R,k)$ denote the chosen isomorphism, and let $\sigma := \varphi(1)$. 
	
	Fix $\xi \in D_{R|k}$. By the definition of $\Phi(\xi)$, we know that for all $f \in R$ we have $\varphi(\Phi(\xi)(f)) = \varphi(f) \xi$ which, by the linearity of $\varphi$, gives
	$$\sigma(\Phi(\xi)(f) g) = \sigma(f \xi(g)) \text{ for all }f, g \in R.$$
	We conclude that, for all $h \in R$,
	$$\sigma(\Phi^2(\xi)(1) h) = \sigma(\Phi(\xi)(h)) = \sigma(h \xi(1));$$
	that is, $\varphi(\Phi^2(\xi)(1)) = \varphi(\xi(1))$. We conclude that $\Phi^2(\xi)(1) = \xi(1)$ for all $\xi \in D_{R|k}$ which, by Lemma \ref{lemma-suff-to-check-on-one}, gives the result. 
\end{proof}

\subsection{A curious result in positive characteristic} \label{subscn-curious-result}

We finish by showing how the generalization of the Skolem-Noether theorem due to Rosenberg and Zelinksi can be used to prove a curious result on uniqueness of anti-automorphisms in characteristic $p$. 

We first state the theorem of Rosenberg and Zelinski. Although their result is much more general, we state it here in a much weaker form that suffices for our purposes (a proof of this weaker statement can also be found in the survey of Isaacs \cite{Isaacs80}). 
\begin{theorem}[{\cite[Thm. 14]{RZ61}}] \label{thm-RZ}
	Suppose that $R$ is a unique factorization domain. Then every $R$-algebra automorphism $\Phi$ of the ring $M_n(R)$ of $n \times n$ matrices over $R$ is inner.
\end{theorem} 
\begin{theorem} \label{thm-curious-charp}
	Let $k$ be a perfect field of characteristic $p>0$ and $R:=k[x_1, \ds, x_n]$ be a polynomial ring over $k$. Then there is a unique anti-automorphism on $D_{R|k}$ that fixes $R$. 
\end{theorem}
\begin{proof}[Proof of Theorem \ref{thm-curious-charp}]
	The existence and an explicit description of such an anti-automorphism is given in Proposition \ref{prop-exists-involn-smoothcase}. Let $\Phi$ and $\Psi$ be two anti-automorphisms on $D_{R|k}$ that fix $R$, and we will show that $\Phi = \Psi$. 
	
	By Lemma \ref{lemma-abstract-stuff}(c), $\Phi$ and $\Psi$ respect the level filtration, and it is enough to show that for every $e>0$ the isomorphisms $\Phi$ and $\Psi$ agree on $D^{(e)}_R$. We thus fix an $e>0$.
	
	The ring $R$ is free as a module over its subring $R^{p^e}$ of $p^e$-th powers, and therefore $D^{(e)}_R = \End_{R^{p^e}}(R)$ is a matrix ring over $R^{p^e}$. Since $R^{p^e}$ is a polynomial ring, and thus a unique factorization domain, Theorem \ref{thm-RZ} applies and, arguing as in the proof of Lemma \ref{lemma-K-involutive-charp}, we conclude that there exists some invertible $\gamma \in R$ such that $\Phi(\xi) = \gamma^{-1} \Psi(\xi) \gamma$ for all $\xi \in D^{(e)}_R$. We now observe that all units of $R$ are in $k$, so $\gamma$ is in $k$ and it therefore commutes with all differential operators. We conclude that $\Phi(\xi) = \Psi(\xi)$ for all $\xi \in D^{(e)}_R$, which proves the result. 
\end{proof}

This result is in stark contrast with the situation in characteristic zero. Indeed, when $\Char k = 0$ and $R = k[x_1, \ds, x_n]$, a presentation for $D_{R|k}$ was given in \ref{subsubscn-rodo-smooth}. If we pick arbitrary elements $f_1, \ds, f_n \in R$ such that for every $i$ the polynomial $f_i$ involves only on the variable $x_i$, then the assignments $[x_i \mapsto x_i], [\partial_i \mapsto - \partial_i + f_i]$ define an involutive anti-automorphism on $D_{R|k}$ that fixes $R$.

\section{Quotient singularities} \label{scn-quot-sing}

In this section we show that if $R$ is the ring of invariants of a linear action of a finite group $G$ with no pseudoreflections on a polynomial ring $S := k[x_1, \ds, x_n]$, where $k$ is a field of characteristic zero, then there exists an involutive anti-automorphism on $D_{R|k}$ that fixes $R$. We begin by recalling the notion of pseudoreflection.

\begin{definition}
	A nontrivial linear transformation $\gamma \in \GL_n(k)$ is a pseudoreflection if there exists a hyperplane $H \sq k^n$ on which $\gamma$ acts trivially.  
\end{definition}

Let $k$ be a field of characteristic zero and $G$ be a finite subgroup of $\GL_n(k)$ that contains no pseudoreflections. The natural action of $G$ on $k^n$ extends to an action by algebra automorphisms on $\Sym_k(k^n)$, which via the natural choice of basis we identify with a polynomial ring $S := k[x_1, \ds, x_n]$. Let $R := S^G$ be the $G$-invariant subring of $S$. 

The group $G$ also acts on the ring $D_{S|k}$ of differential operators on $S$ by the formula
$$(\gamma \cdot \xi)(f) = \gamma \xi (\gamma^{-1} f) \ \ \ (\gamma \in G, \xi \in D_{S|k}, f \in S).$$
It is easy to check that every $G$-invariant differential operator induces a differential operator on $R$, giving a map $D_{S|k}^G \to D_{R|k}$. We have the following theorem of Kantor.

\begin{theorem}[{\cite[III.III]{Kan77}, \cite{Lev81}}] \label{thm-Kantor}
	Let $k, G, S$ and $R$ be as above. Then the natural map $D_{S|k}^G \to D_{R|k}$ is an isomorphism. 
\end{theorem}

From this point onwards, we identify $D_{R|k}$ with the invariant subring $D_{S|k}^G$ of $D_{S|k}$.

\begin{example} \label{example-2-veronese}
	Let $G = \la \sigma| \sigma^2 = 1 \ra $ be a cyclic group of order 2, and we consider the linear action of $G$ on $S := k[s, t]$ given by $\sigma s= - s$, $\sigma t = -t$. Then the invariant subring $R := S^G$ consists of all polynomials that only have homogeneous components of even degree; that is, $R = k[s^2, st, t^2]$. The ring $D_{R|k}$ is then given by
	$$D_{R|k} = k \la s^2, st, t^2, s \partial_s, s \partial_t, t \partial_s, t \partial_t, \partial_s^2, \partial_s \partial_t, \partial_t^2 \ra.$$
	The standard transposition on $D_{S|k}$ (see Subsection \ref{subscn-antiauto-smooth}) preserves the subring $D_{R|k}$, and therefore it restricts to an involutive anti-automorphism of $D_{R|k}$ that fixes $R$.	
\end{example}

We show that, as we just observed in Example \ref{example-2-veronese}, the standard transposition on $D_{S|k}$ always preserves the subring $D_{R|k}$. 
\begin{theorem} \label{thm-quot-invol}
	Let $k$ be a field of characteristic zero, $G$ be a finite subgroup of $\GL_n(k)$ that contains no pseudoreflections, $S := k[x_1, \ds, x_n]$ be a polynomial ring over $k$ equipped with the standard linear action of $G$ and $R := S^G$ be the ring of $G$-invariants of $S$. Then the standard transposition on $D_{S|k}$ preserves the subring $D_{R|k}$, inducing an involutive anti-automorphism on $D_{R|k}$ that fixes the subring $R$. 
\end{theorem}
\begin{proof}
	We note that every ring automorphism of $D_{S|k}$ is also a ring automorphism of $D_{S|k}^{op}$, and therefore the action of $G$ on $D_{S|k}$ is also an action on $D_{S|k}^{op}$. We claim that the standard transposition $\Phi: D_{S|k} \to D_{S|k}^{op}$ is $G$-equivariant; since $D_{R|k}$ is identified with the ring of invariants $D_{S|k}^G$ this will prove the result.
	
	The claim states that for all $\xi \in D_{S|k}$ and all $\gamma \in G$ we have $\gamma \cdot \Phi(\xi) = \Phi(\gamma \cdot \xi)$. As with any action on an algebra, it suffices to check that this holds on algebra generators for $D_{S|k}$. For the generators $x_i$ we have:
	$$\gamma \cdot \Phi(x_i) = \gamma \cdot x_i = x_i = \Phi(x_i) = \Phi(\gamma \cdot x_i).$$
	For the partial derivatives $\partial_i$, we first observe that $\gamma \cdot \Phi(\partial_i) = \gamma \cdot (- \partial_i) = - \gamma \cdot \partial_i$, and therefore it suffices to show that $\Phi(\gamma \cdot \partial_i) = - (\gamma \cdot \partial_i)$. To prove this, we show that $\gamma \cdot \partial_i$ is in the $k$-span of the derivatives $\partial_1, \ds, \partial_n$, i.e. we show that $\gamma \cdot \partial_i \in k \{\partial_1, \ds, \partial_n\}$.
	
	We first note that $k \{\partial_1, \ds, \partial_n\}$ is precisely the set of all derivations $\theta$ with the property that $\theta(x_j) \in k$ for all $j$; indeed, an arbitrary derivation $\theta$ can be written as $\theta = \sum_j f_j \partial_j$ for some $f_j \in S$, and $f_j = \theta(x_j)$. 
	
	For all $f, g \in S$ we have
	\begin{align*}
	(\gamma \cdot \partial_i) (fg) & = \gamma \partial_i (\gamma^{-1}(fg)) \\
		& = \gamma \partial_i ((\gamma^{-1}f) (\gamma^{-1} g)) \\
		& = \gamma \bigg( (\gamma^{-1}f) \partial_i (\gamma^{-1} g) + (\gamma^{-1}g) \partial_i (\gamma^{-1}f) \bigg) \\
		& = f \gamma \big( \partial_i(\gamma^{-1}g) \big) + g \gamma \big( \partial_i (\gamma^{-1}f) \big) \\
		& = f (\gamma \cdot \partial_i) (g) + g (\gamma \cdot \partial_i) (f),
	\end{align*}
	and therefore $\gamma \cdot \partial_i$ is a derivation. Moreover, if we let $a_k \in k$ be such that $\gamma^{-1} x_j = \sum_k a_k x_k$ then
	$$(\gamma \cdot \partial_i)(x_j) = \gamma \big( \partial_i(\gamma^{-1} x_j) \big) = \gamma(a_i) = a_i \in k.$$
	We have therefore shown that $\gamma \cdot \partial_i \in k\{\partial_1, \ds, \partial_n\}$, which concludes the proof. 
\end{proof}
\begin{remark}
	Kantor's theorem (Theorem \ref{thm-Kantor}) remains true whenever $k$ has characteristic $p > 0$ and $p$ does not divide $|G|$ (see Jeffries' notes \cite[\S 4.2]{Jeffries-notes}). However, note that in the proof of Theorem \ref{thm-quot-invol} we use the fact that $D_{S|k}$ is generated by derivations, which crucially uses the fact that $k$ has characteristic zero. Note that the proof also uses the fact that the action of $G$ is linear. 
\end{remark}

\begin{remark}
	Even when $G$ is not finite, there is a map $D_{S|k}^G \to D_R$ which is surjective for some nice group actions \cite{Schwarz95}, and the only obstruction to generalizing Theorem \ref{thm-quot-invol} to other rings of invariants is whether this map is an isomorphism. 
\end{remark}

\bibliography{biblio}
\bibliographystyle{alpha}

\end{document}